\documentclass[hidelinks,12pt,a4paper,reqno]{amsart}
\usepackage{graphicx}
\usepackage[all]{xy}
\usepackage{amssymb}
\usepackage{amsmath}
\usepackage[utf8]{inputenc}
\usepackage{amsthm}

\newtheorem{proposition}{Proposition}
\newtheorem{corollary}{Corollary}

\newtheorem{definition}{Definition}
\newtheorem{theorem}{Theorem}
\newtheorem{example}{Example}

\usepackage[backref=page]{hyperref}
\renewcommand*{\backref}[1]{}
\renewcommand*{\backrefalt}[4]{%
    \ifcase #1 (Not cited.)%
    \or        (Cited on page~#2.)%
    \else      (Cited on pages~#2.)%
    \fi}
\newcommand{\C}{\mathbf{C}}

\newcommand{\Grp}{\mathbf{Grp}}

\newcommand{\Set}{\mathbf{Set}}

\newcommand{\Act}{\mathbf{Act}}

\newcommand{\Mag}{\mathbf{Mag}}
\newcommand{\Smg}{\mathbf{Smg}}

\newcommand{\coker}{\mathrm{coker}}
\newcommand{\Aut}{\mathrm{Aut}}

\begin{document}

\title[On semibiproducts of magmas and semigroups]{On  semibiproducts of magmas and semigroups}


\author{N. Martins-Ferreira}
\address[Nelson Martins-Ferreira]{Instituto Politécnico de Leiria, Leiria, Portugal}
\thanks{ }
\email{martins.ferreira@ipleiria.pt}

\begin{abstract}

A generalization to the categorical notion of biproduct, called semibiproduct, which in the case of groups covers classical semidirect products, has recently been analysed in the category of monoids with surprising results in the classification of weakly Schreier extensions. The purpose of this paper is to extend the study of semibiproducts to the category of semigroups. However, it is observed that  a further analysis into the category of magmas is required in attaining a full comprehension on the subject. Indeed, although there is a subclass of magma-actions, that we call representable, which classify all semibiproducts of magmas whose behaviour is similar to semigroups, it is nevertheless more general than the  subclass of associative magma-actions.


\keywords{semibiproduct, biproduct, semi-direct product, groups, monoids, semigroups, pseudo-action, correction factor, factor system, factor set, extension, perturbation system}


\end{abstract}

\maketitle

\date{Received: date / Accepted: date}

\today

\tableofcontents

\section{Introduction}

An \emph{Ab}-category \cite{MacLane} is a category whose hom-sets are abelian groups. A \emph{biproduct} in an \emph{Ab}-category, introduced by MacLane in his book \emph{Homology} \cite{Homology}, consists of  a diagram
\begin{equation}
\label{diag: biproduct of comm semigroups intro}
\xymatrix{X\ar@<-.5ex>[r]_{i_1} & A\ar@<-.5ex>@{->}[l]_{p_1}\ar@<.5ex>[r]^{p_2} & B \ar@{->}@<.5ex>[l]^{i_2}}
\end{equation}
in which the conditions
\begin{eqnarray}
i_1p_1+i_2p_2=1_A,\quad
p_2i_2=1_B,\quad
p_1i_1=1_{X}\label{eq:biproduct eqs},
\end{eqnarray}
 are satisfied.
 In the \emph{Ab}-category of abelian groups, a biproduct is nothing but a direct product (or a direct sum) and every $A$ is isomorphic to
$X\times B$ which is further isomorphic to $X+B$. This is no longer true in the category of groups
where furthermore there is the issue of  the component-wise addition of group homomorphisms not being necessarily a group homomorphism. In spite of that, the above definition of biproduct {could}, in principle,  be interpreted in the category of groups by requiring that the map associating the element $i_1p_1(a)+i_2p_2(a)$ to every element $a$ in $A$ {would have to be} a group homomorphism and moreover {it would have to be} \emph{the} identity homomorphism on $A$. However, this approach does not take us too far from the abelian world. On the other hand, if dropping the requirement that $p_1$ is a group homomorphism but rather allowing it to be a set-theoretical map (necessarily preserving the neutral element) then we get precisely a semidirect product of groups and suddenly all the classical theory of group split extensions  is at hand (see e.g. \cite{Northcott}, see also Sections \ref{sec: special case of groups} and \ref{sec:discussion}). On the contrary, but as expected, the situation in monoids is more subtle. For example, if the retraction $p_1$ is not necessarily a monoid homomorphism then it does not need to preserve the neutral element. Furthermore, in opposition to its behaviour in groups, the map $p_1$ is not uniquely determined. It is thus surprisingly to observe (\cite{NMF.21}) how  
 the idea  of considering $p_1$ as a map rather than as a homomorphism  is successfully pushed forward into the category of monoids  at the expense of introducing  the new conditions $p_2i_1=0$ and $p_1i_2=0$ to the ones displayed in (\ref{eq:biproduct eqs}). The result is a new characterization of weakly Schreier split extensions of monoids (see also \cite{PFaul,PFaul1}).
 An important observation is that the widely studied case of Schreier split extensions (\cite{DB.NMF.AM.MS.13}) corresponds precisely to the case in which the map $p_1$ is uniquely determined.    The main distinction between our approach (\cite{NMF.21}) and the one taken in \cite{PFaul}  is that a weakly Schreier split extension is classified by a pseudo-action with a correction system thus presenting the object $A$ as a subset of  $X\times B$ rather than as a quotient from it, further details on Section~\ref{sec:discussion}. Finally, there is another particularity which is worth noting when comparing the classical approach to extension theory with the one resulting from extracting the sequence $(i_1,p_2)$ out of diagram  (\ref{diag: biproduct of comm semigroups intro}), namely that  even-though the retraction $p_1$ and the section $i_2$  may fail to be homomorphisms they are nevertheless part of the structure.

At this point it appears  hard to go beyond the context of monoids since the two conditions $p_2i_1=0$ and $p_1i_2=0$, which are easily derived in the case of groups, are essential in order to have the sequence $(i_1,p_2)$ as a monoid extension.
The purpose of this paper is to show that not only it is possible to consider the case of monoids where the two conditions $p_2i_1=0$ and $p_1i_2=0$ are not present, but more remarkably its study can be pushed forward into the context of magmas and semigroups thus leading to some curious results. 

In summary, the main difference between a semibiproduct and a classical biproduct (as already seen in \cite{NMF.21}) is that $p_1$ and $i_2$ are not necessarily homomorphisms. The fact that $p_1$ and $i_2$ are allowed to be maps, not even preserving the neutral element, gives the possibility of interpreting it in the category of semigroups where certainly it is no longer expected that the sequence $(i_1,p_2)$ is an extension. Instead, we find $i_1$ as a kind of equaliser for the homomorphism $p_2$ and the map $p_2i_1p_1$ (Proposition \ref{prop:ker}), whereas $p_2$ is a kind of coequalizer for the identity morphism $1_A$ and the map $i_2p_2$ (Proposition \ref{prop:coker}). As a consequence, if $p_2i_1=0$ then $p_2=\coker(i_1)$ and $i_1=\ker(p_2)$, further details on Section \ref{Sec: kernels and cokernels}. 

While working in the category of magmas rather than at the level of semigroups we get the bigger picture which clarifies some peculiarities that are observed on the passage from monoids to semigroups.

Let us start right away with the more general case so that later on we may observe its particularities. 


\section{The category of magma-actions}

In this section we introduce a category which turns out to be equivalent to the category of semibiproducts of magmas. It is presented as an internal structure in the category of sets and called the category of magma-actions.
Although at a first glance the category of magma-actions is not even vaguely compared with the category of classical group or monoid actions, it will nevertheless be denoted as $\Act$. The reason will become apparent as soon as a group action is seen to be a very particular case of a magma action.

\begin{definition}\label{def: magma-action}
An object in $\Act$ is called a magma-action and it consists of a six-tuple $(X,B,\theta,\varphi,h,t)$ in which $X$ and $B$ are sets, $\theta,\varphi,h,t$ are maps with domain and codomain as displayed
\begin{equation}
\xymatrix{X\times B\times X\times B\ar[r]^(.7){\varphi} & X \ar@<.5ex>[r]^{h} & B \ar@<.5ex>[l]^{t} & B\times B\ar[l]_(.6){\theta}}
\end{equation}
such that for every $x,x'\in X$
\begin{equation}\label{h is morphism}
h(\varphi(x,h(x),x',h(x')))=\theta(h(x),h(x'))
\end{equation}
and if defining a set $R\subseteq X\times B$ as 
\begin{equation}\label{eq: def of R}
R=\{(x,y)\mid \varphi(x,h(x),t(b),b)=x,\, \theta(h(x),b)=b\}
\end{equation}
 then the following conditions are satisfied  for all $x,x'\in X$ and $b,b'\in B$
\begin{eqnarray}
(x,h(x)) \in  R\label{h is well defined}\\
(t(b),b)\in R\label{t is well defined}\\
(x,b),(x',b') \in  R &\Rightarrow& (\varphi(x,b,x',b'),\theta(b,b'))\in R.\label{R is well defined as submagma}
\end{eqnarray}
\end{definition}

For example, if we let $(B,\theta)$ be a magma on the set $B$ and put $X=B$, $h=t=1_B$ and $\varphi(x,b,x',b')=\theta(x,b')$, then the six-tuple $(X,B,\theta,\varphi,h,t)$ is a magma-action if and only if $\theta(b,b)=b$ for every $b\in B$ and we have $R\cong B$.

A morphism in $\Act$, say from an object $(X,B,\theta,\varphi,h,t)$ to an object $(X',B',\theta',\varphi',h',t')$, is a pair of maps $(f,g)$ such that the following diagram is commutative.
\begin{equation*}
\xymatrix{X\times B\times X\times B\ar[d]_{f\times g\times f\times g} \ar[r]^(.7){\varphi} & X\ar[d]_{f} \ar@<.5ex>[r]^{h} & B \ar@<.5ex>[l]^{t}\ar[d]^{g} & B\times B\ar[l]_(.6){\theta} \ar[d]^{g\times g}\\
X'\times B'\times X'\times B'\ar[r]^(.7){\varphi'} & X \ar@<.5ex>[r]^{h'} & B' \ar@<.5ex>[l]^{t'} & B'\times B'\ar[l]_(.6){\theta'}
}
\end{equation*} 

As we will see (Theorem \ref{thm:Act=Sbp(Mag)}) there is an equivalence of categories between $\Act$ and  the category of semibiproducts of magmas, which will be specialized to an equivalence between a full-subcategory of $\Act$ and the category of semibiproducts of semigroups.

For the moment let us observe that $R\subseteq X\times B$ has a magma structure $(R,+)$ determined by $\varphi$ and $\theta$ as
\begin{equation}\label{eq: magma operation on R}
    (x,b)+(x',b')=(\varphi(x,b,x',b'),\theta(b,b')).
\end{equation}
Furthermore, a magma structure can be defined on $X$ as 
\[x+x'=\varphi(x,h(x),x',h(x'))\]
and condition $(\ref{h is morphism})$ is saying that $h\colon{(X,+)\to (B,\theta)}$ is a morphism of magmas. Moreover, due to conditions $(\ref{h is well defined})$ and $(\ref{t is well defined})$ it is easy to see that the diagram
\begin{equation}\label{diag: semibiproduct of a magma action}
\xymatrix{X\ar@<-.5ex>[r]_{\langle 1,h\rangle} & R\ar@<-.5ex>@{->}[l]_{\pi_X}\ar@<.5ex>[r]^{\pi_B} & B \ar@{->}@<.5ex>[l]^{\langle t,1\rangle}}
\end{equation}
 is well defined. Note that $\pi_B$ and $\langle 1,h\rangle$ are morphisms of  magmas whereas $\pi_X$ and $\langle t,1\rangle$ are set-theoretical maps. This shows the existence of a functor from the category of magma-actions to the category of semibiproducts of magmas as detailed in the following section.

\section{Magma-actions and semibiproducts of magmas}

A semibiproduct of magmas is a diagram 
\begin{equation}
\label{diag: semibiproduct of magmas}
\xymatrix{X\ar@<-.5ex>[r]_{k} & A\ar@<-.5ex>@{..>}[l]_{q}\ar@<.5ex>[r]^{p} & B \ar@{->}@<.5ex>@{..>}[l]^{s}}
\end{equation}
in which $X$, $A$ and $B$ are magmas, $k$ and $p$ are magma morphisms while $s$ and $q$ are set theoretical maps, moreover, the following  conditions are satisfied
\begin{equation}\label{eqs: semibiproduct}
kq+sp=1_A,\quad ps=1_B, \quad qk=1_X.
\end{equation}

We may form the category of semibiproducts of magmas by specifying the morphisms between semibiproducts of magmas as triples of magma morphisms, say $(f_1,f_2,f_3)$,  as displayed in the following diagram
\begin{equation}
\label{diag: morphism of semi-biproducts}
\xymatrix{X\ar@<-.5ex>[r]_{k}\ar[d]_{f_1} & A\ar@{->}@<-.0ex>[d]^{f_2}\ar@<-.5ex>@{..>}[l]_{q}\ar@<.5ex>[r]^{p} & B\ar@{->}[d]^{f_3} \ar@{->}@<.5ex>@{..>}[l]^{s}\\
X\ar@<-.5ex>[r]_(.5){k'} & A' \ar@<-.5ex>@{..>}[l]_(.5){q'}\ar@<.5ex>[r]^(.5){p'} & B \ar@{->}@<.5ex>@{..>}[l]^(.5){s'}}
\end{equation}
and such that the whole diagram,  dotted arrows included, is commutative.

\begin{theorem}\label{thm:Act=Sbp(Mag)}
The category of magma-actions is equivalent to the category of semibiproducts of magmas.
\end{theorem}
\begin{proof}
From a semibiproduct of magmas such as the one displayed in $(\ref{diag: semibiproduct of magmas})$ and satisfying the equations $(\ref{eqs: semibiproduct})$, we obtain a magma action defined as
\begin{eqnarray*}
\theta(b,b')&=&b+b'\\
\varphi(x,b,x',b')&=&q((k(x)+s(b))+(k(x')+s(b')))\\
h&=&pk\\
t&=&qs.
\end{eqnarray*}
It is a straightforward calculation to check that it is a magma-action. The relevant observations are as follows:
\begin{eqnarray}
k(x)+sh(x)=k(x)+spk(x)=kq(k(x))+sp(k(x))=k(x)\\
kt(b)+s(b)=kqs(b)+s(b)=kq(s(b))+sp(s(b))=s(b)
\end{eqnarray}
from which by applying $p$ we get
\begin{eqnarray}
h(x)+h(x)=h(x)\\
ht(b)+b=b
\end{eqnarray}
for every $x\in X$ and $b\in B$ so that $$R=\{(x,b)\mid q(k(x)+s(b))=x,\, h(x)+b=b\}.$$
If we allow $k(x)+s(b)$ to be denoted by $a\in A$ and $k(x')+s(b')$ to be denoted by $a'\in A$ while  assuming that both $(x,b)$ and $(x',b')$ are elements in $R$ then  $(q(a+a'),p(a+a'))$ is an element in $R$ as well.  Indeed, since $kq(a+a')+sp(a+a')=a+a'$ we have
\[q(kq(a+a')+sp(a+a'))=q(a+a')\]
and similarly
\[hq(a+a')+p(a+a')=p(kq(a+a')+sp(a+a'))=p(a+a').\]
Conversely, given a magma action we obtain a semibiproduct of magmas as illustrated in diagram $(\ref{diag: semibiproduct of a magma action})$. Note that the set $R$ defined as in $(\ref{eq: def of R})$ consists precisely of those pairs $(x,b)\subseteq X\times B$ for which $$(x,b)=(x,h(x))+(t(b),b)$$ and that $(R,+)$ is a submagma of $X\times B$ which is endowed with a magma operation induced by the maps $\varphi$ and $\theta$. It is routine calculation to check that morphisms of magma-actions correspond to morphisms of semibiproducts and vice-versa. Let us simply observe that the morphism $f_2$ in the diagram $(\ref{diag: morphism of semi-biproducts})$ is completely determined by the morphisms $f_1$ and $f_3$ as
\[f_2(a)=f_2(kq(a)+sp(b))=f_kq(a)+f_2sp(a)=k'f_1q(a)+s'f_3p(a).\]
We have thus a functorial correspondence between magma actions and semibiproducts of magmas. The correspondence is a natural isomorphism on the side of semibiproducts and the identity on the side of magma-actions. The relevant diagram is
\begin{equation}
\label{diag:semi-biproduct bijectio magmas}
\xymatrix{X\ar@<-.5ex>[r]_{k}\ar@{=}[d] & A\ar@{->}@<-.5ex>[d]_{\beta}\ar@<-.5ex>@{..>}[l]_{q}\ar@<.5ex>[r]^{p} & B\ar@{=}[d] \ar@{->}@<.5ex>@{..>}[l]^{s}\\
X\ar@<-.5ex>[r]_(.5){\langle 1,h\rangle} & R\ar@{->}@<-.5ex>[u]_{\alpha} \ar@<-.5ex>@{..>}[l]_(.5){\pi_{X}}\ar@<.5ex>[r]^(.5){\pi_B} & B \ar@{->}@<.5ex>@{..>}[l]^(.5){\langle t,1 \rangle}}
\end{equation}
in which $\alpha(x,b)=k(x)+s(b)$ and $\beta(a)=(q(a),p(a))$.

\end{proof}

The important particular case of semidirect products of unitary magmas, as considered in  \cite{GranJanelidzeSobral}, is obtained if we restrict our attention to those semibiproducts  of the form $(\ref{diag: semibiproduct of magmas})$ in which $pk=0$, $qs=0$, the map $s$ is  a morphism of magmas and the Schreier condition $q(k(x)+s(b))=x$ is satisfied for all $x\in X$ and $b\in B$. 

\begin{corollary}\label{thm: corollary semidirect product of unitary magmas}
The category of semidirect products of unitary magmas is equivalent to the full subcategory of $\Act$ consisting of those magma-actions $(X,B,\theta,\varphi,h,t)$ for which the following conditions are satisfied
\begin{eqnarray}
    h(x)=0\in B\\
    t(b)=0\in X\\
    \theta(0,b)=b=\theta(b,0)\\
    \varphi(0,0,x,b)=x\\
    \varphi(x,b,0,0)=x\\
    \varphi(x,0,0,b)=x\\
    \varphi(0,b,0,b)=0
\end{eqnarray}
for all $x\in X$ and $b\in B$.  
\end{corollary}
\begin{proof}
The last condition is equivalent to the fact that the map $s$ is a morphim of magmas; condition $\varphi(x,0,0,b)=x$ is equivalent to the Schreier condition $q(k(x)+s(b))=x$; the remaining conditions are needed because we are restricting our attention to the category of unitary magmas.
\end{proof}

We remark that there are two reasons for the set $R$ to be a proper subset of $X\times B$, namely that the map $h$ is not always compatible with $\theta$ in the sense that $\theta(h(x),b)=b$, or that the Schreier condition, which can be reformulated as $\varphi(x,h(x),t(b),b)=x$, does not hold for all pairs $(x,b)\in X\times B$. Recall that the set $R$ is defined as in (\ref{eq: def of R}) and use the diagram (\ref{diag:semi-biproduct bijectio magmas}) to translate from magma-actions to semibiproducts and vice-versa. In the case of semidirect products of unitary magmas we have $R$ bijective to $ X\times B$.

\section{Representable magma-actions}\label{sec: representable actions}

Let us now turn our attention to the category whose objects are semi\-biproducts of semigroups rather than semi\-biproducts of magmas. Since the category of semigroups, $\Smg$, is a full subcategory of $\Mag$, the category of magmas, it simply means that we are asking the magmas $X$, $A$ and $B$ of diagram $(\ref{diag: semibiproduct of magmas})$ to be associative. The question is how to characterize the full subcategory of $\Act$ that is equivalent to the category of semibiproducts of semigroups. The result is somehow surprising but we still recover most of the intuition of groups except for the presence of a new ingredient which is invisible in groups but has already  made its appearance  at the level of monoids \cite{NMF.21}.

For a magma-action $(X,B,\theta,\varphi,h,t)$ we introduce the following notation
\begin{eqnarray}
x+x'&=&\varphi(x,h(x),x',h(x'))\label{eq:x+x'}\\
x^b&=&\varphi(x,h(x),t(b),b)\label{eq:x^b}\\
b\cdot x &=& \varphi(b,t(b),x,h(x))\label{eq:b.x}\\
b\times b' &=& \varphi(t(b),b,t(b'),b')\label{eq:bxb'}
\end{eqnarray}
and consider the set $R=\{(x,b)\in X\times B\mid x^b=x,\, \theta(h(x),b)=b\}$ equipped with the magma operation $+$ as in (\ref{eq: magma operation on R}).
\begin{definition}
A magma-action $(X,B,\theta,\varphi,h,t)$ is said to be \emph{representable} if 
\[\varphi(x,b,x',b')=((x+b\cdot x') + \theta(b,h(x'))\times b')^{\theta(b,b')}\]
for every pair $(x,b)$ and $(x',b')$ in the set $R\subseteq X\times B$.
\end{definition}
In other words, a magma-action is representable if the map $\varphi$ is completely determined by its particular cases (\ref{eq:x+x'})--(\ref{eq:bxb'}). Note that the formula above is a generalization of the usual formula for a pseudo-action of groups with a factor system as displayed in Section \ref{sec:discussion}, equation~(\ref{eq:semidirect product factor system}). It also generalizes the formula in the case of monoids, equation (\ref{eq:semibi product of monoids}), and in the case of semigroups, equation (\ref{eq:semibi product of semigroups}), see also Section \ref{sec:further results on semigroups}.

\begin{definition}
A magma-action $(X,B,\theta,\varphi,h,t)$ is said to be \emph{associative} if the magma $(R,+)$ is a semigroup.
\end{definition}

\begin{proposition}\label{thm:every associative magma-action is representable}
Every associative magma-action is representable.
\end{proposition}
\begin{proof}
If $(R,+)$ is a semigroup then $(B,\theta)$ and $(X,+)$ are semigroups too. Hence,
\begin{equation}
\xymatrix{X\ar@<-.5ex>[r]_{\langle 1,h\rangle} & R\ar@<-.5ex>@{..>}[l]_{\pi_X}\ar@<.5ex>[r]^{\pi_B} & B \ar@{..>}@<.5ex>[l]^{\langle t,1\rangle}}
\end{equation}
is a semibiproduct of semigroups and by Theorem \ref{thm:1}, see  Section \ref{sec:further results on semigroups}, we have the desired result from equation (\ref{eq:binary operaration}).
\end{proof}

Clearly, magma-actions are not representable in general and there are representable actions which are not associative as illustrated in the following example.

\begin{example}\label{example:1}
A representable magma-action $(X,B,\theta,\varphi,h,t)$ which is not associative is obtained if we let $X=B=\{1,2\}$,  $\theta(b,b')=2$ when $(b,b')$ is $(1,2)$, otherwise $\theta(b,b')=1$. The map $\varphi$ is displayed in Table~\ref{tab:my_label}
\begin{table}[ht]
    \centering
    \begin{tabular}{c|c|c|c|c|c}
    & $x$ & $b$ & $x'$ & $b'$ & $\varphi(x,b,x',b')$ \\
    \hline
1 & 1 & 1 & 1 & 1 & 1 \\
2 & 2 & 1 & 1 & 1 & 2 \\
3 & 1 & 2 & 1 & 1 & 1 \\
4 & 2 & 2 & 1 & 1 & 2 \\\hline
5 & 1 & 1 & 2 & 1 & 2 \\
6 & 2 & 1 & 2 & 1 & 1 \\
7 & 1 & 2 & 2 & 1 & 1 \\
8 & 2 & 2 & 2 & 1 & 2 \\\hline
9 & 1 & 1 & 1 & 2 & 1 \\
10 & 2 & 1 & 1 & 2 & 2 \\
11 & 1 & 2 & 1 & 2 & 1 \\
12 & 2 & 2 & 1 & 2 & 2 \\\hline
13 & 1 & 1 & 2 & 2 & 2 \\
14 & 2 & 1 & 2 & 2 & 1 \\
15 & 1 & 2 & 2 & 2 & 1 \\
16 & 2 & 2 & 2 & 2 & 2 
    \end{tabular}
    \caption{The map $\varphi\colon X\times B\times X\times B\to X$ of Example \ref{example:1}}
    \label{tab:my_label}
\end{table}
while $h(x)=t(b)=1$ for all $x\in X$ and  for all $b\in B$.

In this case, we obtain
\begin{center}
\begin{tabular}{c|c|c|c|c|c|c|c}
  & $x|b$ & $x'|b'$ & $\theta(b,b')$ & $x+x'$ & $x^{b'}$ & $b\cdot x'$ & $b\times b'$ \\
    \hline
    1 & 1 & 1 & 1 & 1 & 1 & 1 & 1 \\
2 & 2 & 1 & 1 & 2 & 2 & 1 & 1 \\
3 & 1 & 2 & 2 & 2 & 1 & 2 & 1 \\
4 & 2 & 2 & 1 & 1 & 2 & 1 & 1 \\
\end{tabular}
\end{center}
and the set $R=\{(x,b)\mid x^b=x,\theta(h(x),b)=b\}$, bijective to $X\times Y\cong\{11,21,12,22\}$, is equipped with the magma operation
\begin{center}
\begin{tabular}{c|cccc}
    + & 11 & 21 & 12 & 22  \\\hline
11 & 11 & 21 & 12 & 22  \\
21 & 21 & 11 & 22 & 12  \\
12 & 11 & 11 & 11 & 11  \\
22 & 21 & 21 & 21 & 21 
\end{tabular}
\end{center}
which is not associative, e.g. $22+(11+12)=21$ whereas $(22+11)+12=22$. Nevertheless, this magma-action is representable since
\begin{equation*}
    \varphi(x,b,x',b')=((x+b\cdot x') + \theta(b,h(x'))\times b')^{\theta(b,b')}
\end{equation*}
 holds for all $x,x'\in X$ and $b,b'\in B$.
\end{example}

\begin{theorem}
The category of associative magma-actions is equivalent to the category of semibiproducts of semigroups.
\end{theorem}
\begin{proof}
By Theorem \ref{thm:Act=Sbp(Mag)} and the fact that when $(R,+)$ is a semigroup then 
\begin{equation}
\xymatrix{X\ar@<-.5ex>[r]_{\langle 1,h\rangle} & R\ar@<-.5ex>@{..>}[l]_{\pi_X}\ar@<.5ex>[r]^{\pi_B} & B \ar@{..>}@<.5ex>[l]^{\langle t,1\rangle}}
\end{equation}
is a semibiproduct of semigroups. It remains to observe that the magma-action corresponding  to a semibiproduct of semigroups is associative. In particular it is representable as may be seen directly from Theorem~\ref{thm:1}.
\end{proof}

As soon as a magma-action is representable it is preferable to consider it as a pseudo-action as detailed at the end of Section \ref{sec:discussion}.

\section{Semibiproducts of semigroups}\label{sec:further results on semigroups}

In this section we analyse semibiproducts on a deeper level in the category of semigroups and give some remarks on how to generalize it into a wider categorical context. Let us repeat the definition so that no confusion may arise.

\begin{definition}\label{def:1} 
A semibiproduct of semigroups is a diagram 
\begin{equation}
\label{diag: semi-biproduct of semigroups (1)}
\xymatrix{X\ar@<-.5ex>[r]_{k} & A\ar@<-.5ex>@{..>}[l]_{q}\ar@<.5ex>[r]^{p} & B \ar@{->}@<.5ex>@{..>}[l]^{s}}
\end{equation}
in which $X$, $A$ and $B$ are semigroups,  $k$ and $p$ are semigroup homomorphisms, $s$ and $q$ are set theoretical maps and the following conditions are satisfied:
\begin{eqnarray}
kq+sp=1_A\label{eq:kq+sp=1A(1)}\\
ps=1_B\label{eq:ps=1B(1)}\\
qk=1_{X}.\label{eq:qk=1X(1)}
\end{eqnarray}
\end{definition}

Equation (\ref{eq:kq+sp=1A(1)}) has the obvious meaning that $kq(a)+sp(a)=a$ for all $a\in A$. 
The definition of semibiproduct can be generalized into the wider context of a category $\C$, equipped with a bifunctor $H\colon{\C^{\text{op}}\times \C\to\Set}$ and two natural transformations
\[\xymatrix{H\times H \ar[r]^(.6){\mu} & H & \hom_{\C}\ar[l]_{\varepsilon}}.\]
When $\C$ is the category of semigroups then $\hom_{\C}(A,B)$ is the set of semigroup homomorphisms from $A$ to $B$ and we take $H(A,B)$ to be the set of all maps from the underlying set of $A$ into the underlying set of $B$. The natural transformation $\varepsilon$ is the inclusion of a homomorphism as a map while the natural transformation $\mu$ is the usual component-wise addition of maps in semigroups.

In this more general setting, a semibiproduct consists of a sequence
\[\xymatrix{X\ar[r]^{k}&A\ar[r]^{p}&B}\]
together with $s\in H(B,A)$ and $q\in H(A,X)$ such that $\mu(kq,sp)=\varepsilon(1_A)$, $ps=\varepsilon(1_B)$ and $qk=\varepsilon(1_X)$. The notation $guf=H(f,g)(u)$ is useful and particularly successful in expressing $kq=H(1,k)(q)$, $sp=H(p,1)(s)$, $qk=H(k,1)(q)$ and $ps=H(1,p)(s)$. Further details can be found in \cite{Brown,NMF.15,NMF.21}. Moreover, it is worthwhile noting that the category $\Set$ can be replaced by any skew-monoidal category thus resulting a skew-enriched structure in the sense of \cite{Campbell}. A different direction for  approaching in particular Schreier extensions is suggested in \cite{MontoliRodeloLinden}.    


\vspace{.5cm}

Given a semibiproduct of semigroups, say 
 \begin{equation}\label{diag: semi-biproduct of semigroups (2)}
\xymatrix{X\ar@<-.5ex>[r]_{k} & A\ar@<-.5ex>@{..>}[l]_{q}\ar@<.5ex>[r]^{p} & B \ar@{->}@<.5ex>@{..>}[l]^{s}}
\end{equation}
we put $h=pk$ and consider the maps  $\rho$,  $\varphi$, $\gamma$, defined as
\begin{eqnarray}
\rho(x,b)=q(k(x)+s(b))\label{eq:rho}\\
\varphi(b,x)=q(s(b)+k(x))\label{eq:varphi}\\
\gamma(b,b')=q(s(b)+s(b'))\label{eq:gamma}
\end{eqnarray}
for all $x\in X$ and $b,b'\in B$.

The following theorem is a collection of results that are obtained by considering  a semibiproduct of semigroups with $h$, $\rho$, $\varphi$, $\gamma$ as above.

\begin{theorem}\label{thm:1}
Let be given a semibiproduct of semigroups such as the one in (\ref{diag: semi-biproduct of semigroups (2)}). Then:
\begin{enumerate}
\item $h(x)=h(x)+h(x)$, for all $x\in X$;\label{them:1(1)}
\item $p(a)=hq(a)+p(a)$, for all $a\in A$;\label{them:1(2)}
\item $q(a)=\rho(q(a),p(a))$, for all $a\in A$;\label{them:1(3)}
\item the following equation holds for every $a,a'\in A$ 
\begin{equation}
a+a'=k(qa+\varphi(pa,qa')+\gamma(pa+hqa',pa'))+sp(a+a');
\end{equation}\label{them:1(4)}
\item the map $\langle q,p \rangle\colon{A\to X\times B}$ is injective;
\item the image of the map $\langle q,p \rangle$ is $R\subseteq X\times B$ defined as 
\begin{equation}\label{eq:def R}
R=\{(x,b)\mid \rho(x,b)=x,\; h(x)+b=b\};
\end{equation}
\item the map $\alpha\colon{R\to A}$, defined as $\alpha(x,b)=k(x)+s(b)$, is a bijection;
\item the binary operation on the set $X\times B$, defined as
\begin{equation}\label{eq:binary operaration}
(x,b)+(x',b')=(\rho(x+\varphi(b, x')+\gamma(b+h(x'), b'),b+b'),b+b')
\end{equation}
is well defined on the set $R\subseteq X\times B$;
\item the set $R$ equipped with the binary operation (\ref{eq:binary operaration}) is a semigroup;
\item the map $\alpha\colon{R\to A}$ is an isomorphism of semigroups with inverse $\beta\colon{A\to R}$, the map defined as $\beta(a)=(q(a),p(a))$;
\item in the diagram
\begin{equation}
\xymatrix{X\ar@<-.5ex>[r]_{k}\ar@{=}[d] & A\ar@{->}@<-.5ex>[d]_{\beta}\ar@<-.5ex>@{..>}[l]_{q}\ar@<.5ex>[r]^{p} & B\ar@{=}[d] \ar@{->}@<.5ex>@{..>}[l]^{s}\\
X\ar@<-.5ex>[r]_(.5){\iota_X} & R\ar@{->}@<-.5ex>[u]_{\alpha} \ar@<-.5ex>@{..>}[l]_(.5){\pi_{X}}\ar@<.5ex>[r]^(.5){\pi_B} & B \ar@{->}@<.5ex>@{..>}[l]^(.5){\iota_B}}
\end{equation}
where $\pi_B(x,b)=b$, $\pi_X(x,b)=x$, $\iota_B(b)=(qs(b),b)$, $\iota_X(x)=(x,h(x))$, the bottom row is a semibiproduct of semigroups.
\end{enumerate}
\end{theorem}

\begin{proof} We observe:
\begin{enumerate}
\item If starting with $k(x)$ and decomposing it as  $kqk(x)+spk(x)$, which is the same as $k(x)+sh(x)$, then we get
\[h(x)=p(k(x))=p(k(x)+sh(x))=h(x)+h(x).\]

\item For every $a\in A$,
\begin{eqnarray*}
p(a) &=& p(kq(a)+sp(a))\\
 &=& pkq(a)+psp(a)\\
 &=& hq(a)+p(a).
\end{eqnarray*}

\item For every $a\in A$,
\begin{equation*}
q(a)=q(kq(a)+sp(a))=\rho(q(a),p(a)).
\end{equation*}

\item Replacing $\varphi$ and $\gamma$ in the equation we have
\[a+a'=k(qa+q(spa+kqa')+q(s(pa+hqa')+spa'))+sp(a+a')\]
which is obtained as
\begin{eqnarray*}
a+a'&=& (kqa+spa)+(kqa'+spa'),  \\
    &=& kqa+(spa+kqa')+spa', \quad (\text{let $u=spa+kqa'$})\\
    &=& kqa+(kq(u)+sp(u))+spa' \\
    &=& kqa+kq(u)+(sp(u)+spa') \\
    &=& kqa+kq(u)+kq(sp(u)+spa')+sp(sp(u)+spa') \\
    &=& kqa+kq(u)+kq(s(pa+hqa')+spa')+s(p(u)+pa') \\
    &=& kqa+kq(u)+kq(s(pa+hqa')+spa')+sp(u+a') \\
    &=& k(qa+q(u)+q(s(pa+hqa')+spa'))+sp(a+a')  
\end{eqnarray*}
with $p(u+a')$ being the same as $p(a+a')$ due to $u=spa+kqa'$ and $p(spa+kqa'+a')=pa+p(kqa'+spa')=pa+pa'$.

\item The map $a\mapsto(q(a),p(a))$ is injective for if $(qa,pa)=(qa',pa')$ then $a=kqa+spa=kqa'+spa'=a'$.

\item If $a\in A$ then $(qa,pa)\in R$; as a matter of fact we have already seen that $\rho(q(a),p(a))=q(a)$ and $hq(a)+p(a)=p(a)$. Similarly, if $(x,b)\in R$ then there exists $a\in A$, namely $a=kx+sb$, with $q(a)=x$ and $p(a)=b$. Indeed, because $(x,b)\in R$ we have $q(a)=q(kx+sb)=\rho(x,b)=x$ and $p(a)=p(kx+sb)=h(x)+b=b$.

\item On the one hand we have
\[a\mapsto (qa,pa)\mapsto kqa+spa=a\] 
while on the other hand
\[(x,b)\mapsto kx+sb\mapsto(q(kx+sb),p(kx+sb))=(\rho(x,b),h(x)+b)\]
and if $(x,b)\in R$ then $(\rho(x,b),h(x)+b)=(x,b)$.

\item Considering $a=kx+sb$ and $a'=kx'+sb'$ we have $(q(a+a'),p(a+a'))\in R$; by the item (4) in the Theorem we know that $(q(a+a'),p(a+a'))$ is precisely \[(\rho(x+\varphi(b, x')+\gamma(b+h(x'), b'),b+b'),b+b')\]
as soon as $(x,b)$ and $(x',b')$ are both in $R$.

\item Since $\alpha$ is a bijection and the operation in $R$ is obtained as $(x,b)+(x',b')=(q(a+a'),p(a+a'))$ with $a+a'=kx+sb+kx'+sb'$ it follows that it must be associative;

\item and $\alpha$ is an isomorphism with inverse $\beta(a)=(q(a),p(a))$.

\item We have $\pi_B\iota_B=1_B$ and $\pi_X\iota_X=1_X$.  In order to prove 
\[\iota_X\pi_X+\iota_B\pi_B=1_R\]
 first we observe that the identities
\begin{eqnarray}
q\alpha=\pi_X\\
\beta k=\iota_X\\
p\alpha=\pi_B\\
\beta s=\iota_B
\end{eqnarray}
hold true and then we compute
\[1_R=\beta\alpha=\beta(kq+sp)\alpha=\beta kq\alpha+\beta s p \alpha=\iota_X\pi_X+\iota_B\pi_B.\]
\end{enumerate}
This shows that the bottom row in the diagram is a semibiproduct of semigroups as desired.
\end{proof}

\section{The special case of groups}\label{sec: special case of groups}

In this section we provide a survey analysis in the case of groups and compare it with classical results.  
It is not difficult to see that in groups the well-known correspondence between semidirect products and split extensions is expanded into a correspondence between semibiproducts and extensions with a specified section (but the section need not be a homomorphism). This is similar to the original theory developed by Schreier and MacLane on the classification of nonabelian extensions of abstract groups which has led to low dimensional nonabelian group cohomology, sometimes called Schreier’s theory of nonabelian group extensions (\cite{EM,MacLane2,LB}). Let us recall that
the traditional Schreier-MacLane way to obtain a nonabelian group 2-cocycle from a group extension  starts with choosing a set-theoretic section of the quotient homomorphism $p:A\to B$. In our case the choice of the section is not needed because in a semibiproduct it is already part of the structure.

We will denote a semibiproduct of semigroups such as the one displayed in (\ref{diag: semi-biproduct of semigroups (2)}) as a tuple $(X,A,B,p,q,k,s)$ and associate to it the tuple $(h,\rho,\varphi,\gamma)$ with $h=pk$ and $\rho,\varphi,\gamma$ the maps defined as in (\ref{eq:rho})--(\ref{eq:gamma}).

\begin{proposition}
Let $p\colon{A\to B}$ be a surjective group homomorphism with a specified section map $s\colon{B\to A}$, i.e., $ps=1_B$. Then the tuple $(X,A,B,p,q,k,s)$ is a semibiproduct of groups as soon as $k\colon{X\to A}$ is the kernel of $p$ and the map $q\colon{A\to X}$ is such that $kq(a)=a-sp(a)$, for all $a\in A$. Moreover, if the section $s$ is a group homomorphism then the semibiproduct $(X,A,B,p,q,k,s)$ is a semidirect product.
\end{proposition}

\begin{proposition}\label{prop:1}
Let $(X,A,B,p,q,k,s)$ be a semibiproduct of semigroups with associated tuple $(h,\rho,\varphi,\gamma)$  as displayed in (\ref{diag: semi-biproduct of semigroups (2)}). If $X$, $A$ and $B$ are groups, then:
\begin{enumerate}
\item the map $q$ is uniquely determined as $kq(a)=a-sp(a)$, for all $a\in A$;

\item $h=pk$ is the trivial homomorphism;

\item the map $\rho$ is uniquely determined as $\rho(x,b)=x$ for all $(x,b)\in X\times B$;

\item $k$ is the kernel of $p$;

\item $p$ is the cokernel of $k$;

\item the maps $\varphi$ and $\gamma$ encode the information of a pseudo action with a factor system and $A$ is isomorphic to the group $X\rtimes_{\varphi,\gamma} B$ whose operation is
\[(x,b)+(x',b')=(x+\varphi(b,x)+\gamma(b,b'),b+b');\]

\item the maps $\varphi$ and $\gamma$ encode the information of a normal pseudo-functor
\[F\colon{B\to \Grp}\]
with $F(b)\colon{X\to X}$ as $F(b)(x)=\varphi(b,x)$ and $$F_{b,b'}\colon{F(b+b')\Longrightarrow F(b)F(b')}$$ as $F_{b,b'}=\gamma(b,b')\in X$.
\end{enumerate}
\end{proposition}
\begin{proof}
We observe:
\begin{enumerate}
\item follows from $kq(a)+sp(a)=a$ for all $a\in A$;
\item follows from Theorem \ref{thm:1}(\ref{them:1(1)}); 
\item using the fact that $h=pk=0$ we have
\begin{eqnarray*}
kx+sb  &=&  kq(kx+sb)+sp(kx+sb)\\
&=&  kq(kx+sb)+s(pkx+psb)\\
&=&  kq(kx+sb)+s(hx+b)\\
&=&  kq(kx+sb)+s(b)
\end{eqnarray*}
and cancelling out $s(b)$ on both sides we obtain $k(x)=kq(k(x)+s(b))$ from which we conclude $\rho(x,b)=x$;
\item follows from Corollary \ref{cor:ker}, see Section \ref{Sec: kernels and cokernels};
\item follows from Corollary \ref{cor:coker}, see Section \ref{Sec: kernels and cokernels};
\item it is a classical result from Schreier theory, see for example \cite{Northcott};
\item it is a classical result that makes use of the  Grothendieck construction, see \cite{GManuell} for a recent developement on that direction; it considers $B$ as a one object groupoid and the functor $F$ takes values in $\Grp$, the category of all groups, and sends the unique object in the groupoid $B$ to the group $X$ in $\Grp$. 
\end{enumerate}
\end{proof}

%
%
%
%

Summing up, in the context of groups, from every semibiproduct $(X,A,B,p,q,k,s)$ we extract a group extension $X\to A\to B$ with associated maps $\varphi$ and $\gamma$ respectively as pseudo-action and factor system. On the other hand, every group extension with a specified section gives rise to a semibiproduct. However, the same extension
\[\xymatrix{X\ar[r]^{k} & A\ar[r]^{p} & B}\]
 if considered with different sections may give inequivalent semi\-bi\-pro\-ducts due to the fact that an isomorphism between semi\-biproducts has to be compatible with the maps $q$ and $s$ whereas an isomorphism of extensions need not be.


\vspace{.5cm}

\section{Semibiproducts of semigroups as extensions}\label{Sec: kernels and cokernels}

Let us now analyse the general case of semigroups and see how to deal with the similar notions of kernel and cokernel without assuming the existence of a null object.
An important aspect to keep in mind is that the equality
\[(kq+sp)f=kqf+spf\] 
is always true, even when $f$ is a map, whereas  in order to ensure that
\[f(kq+sp)=fkq+fsp\]
 we should require $f$ to be a  homomorphism.
 
 We start by looking at a notion similar to a cokernel of semigroups with respect to a semibiproduct diagram.
 
\begin{proposition}\label{prop:coker}
Let $(X.B,A,p,q,k,s)$ be a semibiproduct of semigroups. For every semigroup homomorphism $f\colon{A\to Z}$, with $f=fsp$, there exists a unique semigroup homomorphism $\bar{f}\colon{B\to Z}$ such that $f=\bar{f}p$. 
\end{proposition}
\begin{proof}
The map $fs$ is a homomorphism
\[fs(b+b')=fs(psb+psb')=fsp(sb+sb')=f(sb+sb')=fs(b)+fs(b')\]
and so $\bar{f}=fs$ is one solution. To prove uniqueness we observe that if $\bar{f}$ is such that $\bar{f}p=f$ then
\[\bar{f}=\bar{f}ps=fs.\]
\end{proof}
\begin{corollary}\label{cor:coker}
In monoids, if $pk=0$ then $p$ is the cokernel of $k$.
\end{corollary}
\begin{proof}
If $fk=0$ then $f=f(kq+sp)=fkq+fsp=fsp$, and using the previous proposition we conclude that there exists a unique $\bar{f}$ such that $\bar{f}p=f$. Moreover, when $pk=0$, the condition $f=fsp$ is equivalent to $fk=0$. Indeed, if $f=fsp$ then $fk=fspk=0$.
\end{proof}

Contrary to the case of groups, in monoids, we can have a semibiproduct in which $h=pk$ is not the trivial homomorphism. Take $A$ to be any monoid of idempotents, that is, $a=a+a$ for all $a\in A$, then the diagram
\[\xymatrix{A\ar@<-.5ex>[r]_{1_A} & A\ar@<-.5ex>@{..>}[l]_{1_A}\ar@<.5ex>[r]^{1_A} & A \ar@{->}@<.5ex>@{..>}[l]^{1_A}}\]
is a semibiproduct and $h=1_A$.

\vspace{.5cm}
Let us now investigate a  notion similar to a kernel of semigroups with respect to a semibiproduct diagram. In this case, even if $h=pk$ is not the trivial homomorphism we observe that $pf=hqf$ implies $f=kqf$.

\begin{proposition}\label{prop:ker}
Let $(X.B,A,p,q,k,s)$ be a semibiproduct of semigroups. For every semigroup homomorphism $f\colon{Z\to A}$, with $pf=hqf$, there exists a unique semigroup homomorphism $\bar{f}\colon{Z\to X}$ such that $f=k\bar{f}$. 
\end{proposition}
\begin{proof}
Firstly we observe that if $pf=hqf$ then $f=kqf$. Indeed,
\begin{eqnarray*}
f=(kq+sp)f=kqf+spf=kqf+shqf
\end{eqnarray*}
and having in mind that $h$ is $pk$ and that $qk=1_X$ we have
\begin{eqnarray*}
f=kqkqf+spkqf=(kq+sp)kqf=kqf.
\end{eqnarray*}
 Secondly, we observe that $kqf$ being a homomorphism implies that $qf$ is a homomorphism too. Consequently $\bar{f}=qf$ is one solution. To prove uniqueness we observe that if $\bar{f}$ is such that $k\bar{f}=f$ then $k\bar{f}=f=kqf$ which implies $qk\bar{f}=qkqf$ or $\bar{f}=qf$.
\end{proof}
\begin{corollary}\label{cor:ker}
In monoids, if $pk=0$ then $k$ is the kernel of $p$.
\end{corollary}
\begin{proof}
When $pk=0$ the previous proposition asserts precisely that $k$ is the kernel of $p$.
\end{proof}

We conclude this section with a list of all semibiproducts of semigroups with fixed ends of order 2 whose middle object is of order 3. 

Let $B$ and $X$ be the same two element set, say $\{1,2\}$, and consider it with the four possible semigroup structures (up to equivalence) represented by the following multiplication tables
\begin{equation}
M_1=\begin{bmatrix}
1 & 1 \\ 1 & 2 
\end{bmatrix}
,\,
M_2=\begin{bmatrix}
1 & 1 \\ 2 & 2 
\end{bmatrix}
,\,
M_3=\begin{bmatrix}
1 & 2 \\ 2 & 1 
\end{bmatrix}
,\,
M_4=\begin{bmatrix}
1 & 1 \\ 1 & 1 
\end{bmatrix}.
\end{equation}
The following table displays the number of semibiproducts (in which the middle semigroup is of order 3) with ends $X_i=M_i$ and $B_j=M_j$ for all possible cases $i,j=1,2,3, 4$. 
\begin{center}
\begin{tabular}{c|c|c|c|c|}
 & $B_1$ & $B_2$ & $B_3$ & $B_4$ \\ 
\hline 
$X_1$ & 2 & 0 & 2 & 0 \\ 
\hline 
$X_2$ & 4 & 0 & 0 & 0 \\ 	
\hline 
$X_3$ & 2 & 0 & 0 & 0 \\ 
\hline 
$X_4$ & 0 & 0 & 0 & 0 \\ 
\hline 
\end{tabular} 
\end{center}
Here is a detailed  description for each case:
\begin{enumerate}
\item The two cases with ends $X_1$ and $B_1$ have the same maps $p,q,k,s$, defined as
\begin{center}
\begin{tabular}{c|c|c}
$a\in A$ & $p(a)$ & $q(a)$\\ 
\hline
1 & 1 & 1  \\ 
2 & 2 & 2  \\ 
3 & 1 & 2  
\end{tabular} 
\quad
\begin{tabular}{c|c|c}
$x|b$ & $k(x)$ & $s(b)$\\ 
\hline
1 & 1 & 3  \\ 
2 & 2 & 2  
\end{tabular} 
\end{center}
and two different multiplication tables for the middle object.
\begin{equation*}
A_1=\begin{bmatrix}
1 & 1 & 1\\
1 & 2 & 3\\
1 & 3 & 1
\end{bmatrix}
\quad
A_2=\begin{bmatrix}
1 & 1 & 1\\
1 & 2 & 3\\
3 & 3 & 3
\end{bmatrix}
\end{equation*}

\item The two cases with ends $X_1$ and $B_3$ have the same maps $p,k,s$, and the same multiplication table for the middle object $A$,
\begin{center}
\begin{tabular}{c|c}
$a\in A$ & $p(a)$ \\ 
\hline
1 & 1   \\ 
2 & 1   \\ 
3 & 2   
\end{tabular} 
\quad
\begin{tabular}{c|c|c}
$x|b$ & $k(x)$ & $s(b)$\\ 
\hline
1 & 1 & 2  \\ 
2 & 2 & 3  
\end{tabular}
\quad 
$A=\begin{bmatrix}
1 & 1 & 1\\
1 & 2 & 3\\
1 & 3 & 1
\end{bmatrix}$
\end{center}
and admit two different maps $q_1,q_2$ as tabulated.
\begin{center}
\begin{tabular}{c|c|c}
$a\in A$ & $q_1(a)$ & $q_2(a)$\\ 
\hline
1 & 1 & 1  \\ 
2 & 2 & 2  \\ 
3 & 1 & 2  
\end{tabular} 
\end{center}

\item The four cases with ends $X_2$ and $B_1$ have the same maps $p,k$, and the same multiplication table for the middle object $A$,
\begin{center}
\begin{tabular}{c|c}
$a\in A$ & $p(a)$ \\ 
\hline
1 & 2   \\ 
2 & 2   \\ 
3 & 1   
\end{tabular} 
\quad
\begin{tabular}{c|c}
$x$ & $k(x)$ \\ 
\hline
1 & 1  \\ 
2 & 2  
\end{tabular}
\quad 
$A=\begin{bmatrix}
1 & 1 & 1\\
1 & 2 & 3\\
1 & 3 & 1
\end{bmatrix}$
\end{center}
and admit two different possibilities for the map $s$ combined with two different possibilities for the map $q$ as tabulated.
\begin{center}
\begin{tabular}{c|c|c}
$a\in A$ & $q_1(a)$ & $q_2(a)$\\ 
\hline
1 & 1 & 1  \\ 
2 & 2 & 2  \\ 
3 & 1 & 2  
\end{tabular} 
\quad
\begin{tabular}{c|c|c}
$b$ & $s_1(b)$ & $s_2(b)$\\ 
\hline
1 & 3 & 3  \\ 
2 & 1 & 2  
\end{tabular}
\end{center}

\item The two cases with ends $X_3$ and $B_1$ have the same maps $p,k,s$, and the same multiplication table for the middle object $A$,
\begin{center}
\begin{tabular}{c|c}
$a\in A$ & $p(a)$ \\ 
\hline
1 & 2   \\ 
2 & 2   \\ 
3 & 1   
\end{tabular} 
\quad
\begin{tabular}{c|c|c}
$x|b$ & $k(x)$ & $s(b)$\\ 
\hline
1 & 1 & 3  \\ 
2 & 2 & 1  
\end{tabular}
\quad 
$A=\begin{bmatrix}
1 & 2 & 3\\
2 & 1 & 3\\
3 & 3 & 3
\end{bmatrix}$
\end{center}
and admit two different maps $q_1,q_2$ as tabulated.
\begin{center}
\begin{tabular}{c|c|c}
$a\in A$ & $q_1(a)$ & $q_2(a)$\\ 
\hline
1 & 1 & 1  \\ 
2 & 2 & 2  \\ 
3 & 1 & 2  
\end{tabular} 
\end{center}

\end{enumerate}







\section{Brief discussion with a survey of results}\label{sec:discussion}

Recall that in an \emph{Ab}-category, a diagram such as 
\begin{equation}
\label{diag: biproduct of comm semigroups}
\xymatrix{X\ar@<-.5ex>[r]_{k} & A\ar@<-.5ex>@{->}[l]_{q}\ar@<.5ex>[r]^{p} & B \ar@{->}@<.5ex>[l]^{s}}
\end{equation}
satisfying the conditions
\begin{equation}\label{eq: semibiproduct conditions}
    kq+sp=1_A,\quad ps=1_B,\quad qk=1_X
\end{equation}
is simultaneously a product and a coproduct. In particular, there exists a null object and the identities $pk=0$ and $qs=0$ are derived \cite{MacLane}. Clearly, this is not something that can be expected in the category of commutative magmas (or semigroups), nor even in the category of commutative monoids. However, contrary to the case of magmas and semigroups, in monoids the two extra  conditions $pk=0$ and $qs=0$ can be  included as part of the definition thus giving rise to a pointed semibiproduct (see \cite{NMF.21}). This is not possible in commutative semigroups due to the lack of a null object. Surprisingly, as we have seen, there is a way to work out the notion of biproduct of commutative semigroups. Even more surprisingly is the fact that commutativity can be dropped as sson as  we allow set-theoretical maps to enter into our diagrams alongside with homomorphisms. The result is a richer notion of semibiproduct of magmas and semigroups, not necessarily commutative. As a consequence we obtain a classification of semibiproducts of magmas (Theorem  \ref{thm:Act=Sbp(Mag)}) in terms of magma-actions from which  several interesting particular cases are derived. Such is the case of semidirect products of unitary magmas  illustrated in Corollary~\ref{thm: corollary semidirect product of unitary magmas}, or the case of representable actions which include all associative magma-actions (Proposition \ref{thm:every associative magma-action is representable}). 
 
 The particular case of groups has been  analysed with some detail in Section \ref{sec: special case of groups}.
It was observed that a classical semidirect product of groups can be seen as a diagram 
\begin{equation}
\label{diag: semi-direct product of groups}
\xymatrix{X\ar@<-.5ex>[r]_{k} & A\ar@<-.5ex>@{..>}[l]_{q}\ar@<.5ex>[r]^{p} & B \ar@{->}@<.5ex>[l]^{s}}
\end{equation}
in which $X$, $A$ and $B$ are groups (not necessarily abelian),  $k$, $s$, $p$ are group homomorphisms while $q$ is a set theoretical map, moreover, conditions (\ref{eq: semibiproduct conditions}) are satisfied. It follows that the group $A$ is isomorphic to a group $X\rtimes_{\varphi}B$, called the semidirect product of $X$ and $B$ via the action 
\[\varphi\colon{B\to\Aut(X)},\]
  obtained as $\varphi(b)(x)=b\cdot x=q(s(b)+k(x))$, whose group operation is
  \begin{equation}\label{eq:semidirect product}
  (x,b)+(x',b')=(x+b\cdot x',b+b').
  \end{equation}

Our aim was to study the notion of semidirect product in the case of semigroups while extending it in the direction of a biproduct, hence the name \emph{semibiproduct}. It turns out that in order to better understand semibiproducts of semigroups we have to further analyse semibiproducts of magmas.

As we have seen, a semibiproduct of magmas is a diagram 
\begin{equation}
\label{diag: semi-biproduct of semigroups}
\xymatrix{X\ar@<-.5ex>[r]_{k} & A\ar@<-.5ex>@{..>}[l]_{q}\ar@<.5ex>[r]^{p} & B \ar@{->}@<.5ex>@{..>}[l]^{s}}
\end{equation}
in which $X$, $A$ and $B$ are magmas (not necessarily commutative),  $k$ and $p$ are magma homomorphisms while $s$ and $q$ are set theoretical maps, moreover, conditions (\ref{eq: semibiproduct conditions}) are satisfied.

When $X$, $A$ and $B$ are groups the distance from the map $s$ of being a homomorphism is well understood (see e.g. \cite{MacLane2} and its references to previous work). In that case the group $A$ is isomorphic to a group $X\rtimes_{\varphi,\gamma}B$, called the semidirect product of $X$ and $B$ via the pseudo-action 
\[\varphi(b,x)=b\cdot x=q(s(b)+k(x))=s(b)+k(x)-s(b)\]
 and the factor system $\gamma\colon{B\times B\to B}$
 \[\gamma(b,b')=b\times b'=q(s(b)+s(b'))=s(b)+s(b')-s(b+b'),\]
 whose group operation is given by the formula
  \begin{equation}\label{eq:semidirect product factor system}
  (x,b)+(x',b')=(x+(b\cdot x')+(b\times b'),b+b').
  \end{equation}

The case when $X$, $A$ and $B$ are monoids, $k$ and $p$ are monoid homomorphisms and the set theoretical maps $s$ and $q$ preserve the neutral element is quite different from the case of groups (see \cite{NMF.21}, see also \cite{NMF14,Fleischer,Leech,Wells}). Firstly, the extra conditions $pk=0$ and $qs=0$ have to be imposed. Secondly, the monoid $A$ is no longer isomorphic to a monoid $X\rtimes_{\varphi,\gamma} B$ whose underlying set is the cartesian product $X\times B$. As proved in \cite{NMF.21}, in the case of monoids, there is a new ingredient which is invisible in groups. This new ingredient is called a \emph{correction system} in \cite{NMF.21} and it consists of a map $\rho\colon{X\times B\to X}$ denoted by $\rho(x,b)=x^b$ and obtained as
 \[\rho(x,b)=q(k(x)+s(b)).\]
  It can be proved that the correction system is trivial, i.e. $\rho(x,b)=x$, as soon as the monoid $X$ admits cancellation on the right and $B$ is a group. 
The correction system $\rho$ must satisfy some conditions together with the factor system $\gamma$ and the map $\varphi$, which is no longer an action --- it was called a pre-action in \cite{NMF.21}. With these three ingredients at hand we are able to recover the monoid $A$ as being isomorphic to a subset of the cartesian product $X\times B$, namely $R\subseteq X\times B$ defined  as\footnote{Note that we are using $x^b$ as $\rho(x,b)$ in the same way as it is customary to use $b\cdot x$ as $\varphi(b,x)$.}
$$(x,b)\in R\Leftrightarrow x^b=x$$
with neutral element $(0,0)\in R$ and the operation 
 \begin{equation}\label{eq:semibi product of monoids}
  (x,b)+(x',b')=((x+(b\cdot x')+(b\times b'))^{b+b'},b+b'),
  \end{equation}
which can be shown to be well defined on $R$ and associative there. We often write $R$ as $R_{\rho,\varphi,\gamma}$ and thus $A\cong R_{\rho,\varphi,\gamma}$. Clearly, when $\rho(x,b)=x$ is the trivial correction system then we have the same result as for groups with the difference that the extension
$$\xymatrix{X\ar[r]^{k}& A \ar[r]^{p} & B,}$$
is a Schreier extension (see \cite{DB.NMF.AM.MS.13,NMF et all,NMF.21}), rather than an arbitrary extension. Indeed, asking the correction system to be trivial is the same as asking $q(k(x)+s(b))=x$, which is precisely the Schreier condition considerer  in \cite{DB.NMF.AM.MS.13}. Furthermore, in groups, the map $q$ is uniquely determined as $q(a)=a-sp(a)$, while in monoids it is uniquely determined provided the extra conditions $q(k(x)+s(b))=x$ and $pk=0$ are satisfied. When that is the case, the map $q$ is uniquely determined as the $X$-component for the inverse map of $\alpha\colon{X\times B\to A}$, defined as $\alpha(x,b)=k(x)+s(b)$.

\vspace{.5cm}

The results obtained in \cite{NMF.21} for the context of monoids were extended here into the context of magmas and semigroups.
The case of semigroups is even more surprising when compared to groups than the case of monoids. In semigroups, even-though the correction system may be trivial, $x^b=x$, it does not follow that the semigroup $A$ is isomorphic to a semigroup whose underlying set is the cartesian product $X\times B$. This new phenomenon is explained by the lack of condition $pk=0$ which creates a new homomorphism, $h\colon{X\to B}$, as $h=pk$. This new ingredient is used in a refinement of the subset $R=R_{h,\rho,\varphi,\gamma}\subseteq X\times B$ via the formula
\begin{equation}\label{eq:R rho h}
(x,b)\in R\Leftrightarrow \rho(x,b)=x,\; h(x)+b=b.
\end{equation}

A simple example that illustrates the situation is obtained by considering $X=B=(\{0,1\},\cdot)$, the semigroup with the cardinal numbers 0 and 1  and the usual multiplication between them. With all other ingredients being trivial, that is $\rho(x,b)=x$, $\varphi(b,x)=x$ and $\gamma(b,b')=1$ for all $x$ and $b$, there are still three different homomorphisms $h\colon{X\to B}$ to be considered. The two constant maps, $h(x)=0$, $h(x)=1$ and the identity map $h(x)=x$. As expected, when $h(x)=1$ then $R_1=X\times B$ is the cartesian product semigroup. When $h$ is the identity homomorphism then $R_h=\{(0,0),(1,0),(1,1)\}$ is a subsemigroup of the cartesian product $X\times B$. When $h(x)=0$ we obtain $R_0=\{(0,0),(1,0)\}$ which is itself a subsemigroup of $R_h=\{(0,0),(1,0),(1,1)\}$ and isomorphic to $X$,
\[X\cong R_0\hookrightarrow R_h\hookrightarrow R_1=X\times B.\]
There is one more aspect in which the homomorphism $h=pk$ makes an unexpected appearance when compared with the situation in monoids. That is the formula (\ref{eq:semibi product of monoids}) has to be modified to become
 \begin{equation}\label{eq:semibi product of semigroups}
  (x,b)+(x',b')=((x+(b\cdot x')+((b+h(x'))\times b'))^{b+b'},b+b').
  \end{equation}

We have seen  that in every semibiproduct of semigroups such as in (\ref{diag: semi-biproduct of semigroups}), the semigroup $A$ is isomorphic to the set $R=R_{h,\rho,\varphi,\gamma}$ defined as in (\ref{eq:R rho h}) with the binary operation (\ref{eq:semibi product of semigroups}).
However, the information carried out by a pseudo-action of semigroups, in the sense of a homomorphism $h\colon{X\to B}$ together with a correction system $\rho\colon{X\times B\to X}$, a pre-action $\varphi\colon{B\times X\to X}$ and a factor system $\gamma\colon{B\times B\to X}$, such that $(R_{h,\rho,\varphi,\gamma},+)$ is a semigroup, is by itself not sufficient to recover the complete structure on the  semibiproduct diagram. Indeed, the map $t=qs$ is missing.

\section{Conclusion}

In this paper we have given the first steps towards a  theory of semibiproducts by mimicking the classical theory of biproducts in \emph{Ab}-categories. In particular we have seen how it can be applied in the study of extensions even though null objects are not present. However, the theory can be applied in other situations as well. For example in the study of preordered groups and preordered monoids \cite{Preord,NMFMS}  instead of maps we can take monotone maps whereas in the study of topological semigroups \cite{Ganci} we can take continuous maps. Classical algebras, Lie algebras, Hopf algebras and similar structures can be analysed too.

\section*{Acknowledgements}



This work is supported by Fundação para a Ciência e a Tecnologia FCT/MCTES (PIDDAC) through the following Projects:  Associate Laboratory ARISE LA/P/\-0112/2020; UIDP/04044/2020; UIDB/04\-044/2020; PAMI - ROTEIRO/0328/2013 (Nº 022158); MATIS (CEN\-TRO-01-0145-FEDER-000014 - 3362); Generative.Thermodynamic; by CDRSP and ESTG from the Polytechnic Institute of Leiria.

\end{document}